\documentclass{amsart}
\usepackage{graphicx}
\usepackage{interval}
\intervalconfig{soft open fences}
\newcommand{\C}{\mathbb{C}}
\newcommand{\Cn}{\mathbb{C}^n}
\newcommand{\F}{\mathcal{F}}
\newtheorem{thm}{Theorem}[section]
\newtheorem{cor}[thm]{Corollary}
\newtheorem{lem}[thm]{Lemma}

\theoremstyle{definition}

\theoremstyle{remark}
\newtheorem{rem}[thm]{Remark}
\numberwithin{equation}{section}

\newcommand{\abs}[1]{\left\vert#1\right\vert}
\newcommand{\set}[1]{\left\{#1\right\}}

\begin{document}
    \title{{Schottky's theorem in $\C^n$}}
    \author{{P.V.Dovbush}}
    \address{}%
    \email{}%


    \begin{abstract}
        The aim of this note is to give a proof of the Schottky theorem in general domains in $\Cn$.
        The proof is short and works for the cases $n = 1$ and $n > 1$ at the same time.
    \end{abstract}
    \maketitle

    Let $\Omega$ be a bounded domain in $\Cn$.
        A family $\mathcal{F}\subset H(\Omega)$ is normal in $\Omega$ if every sequence of functions $\{f_j\}\subseteq F$ contains either a subsequence which converges to a limit function $f$ uniformly on each compact subset of $\Omega$, or a subsequence which converges uniformly to $\infty$ on each compact subset.

    The notion of a normal family has played an important role in the development of complex function theory.
    We refer to \cite{[RBB]} and the references therein for various results on normal families and applications in case $n=1$.
    A series of important theorems surround the Montel criterion for normality, thereby constituting the so-called Montel cycle.
    The famous Schottky theorem is one of them.

    Scottky's Theorem can be given a completely elementary proof, based on nothing more than Montel's Theorem and general properties of normal families, which works for the cases $n = 1$ and $n > 1$ at the same time.

    \begin{thm}
        \label{shottkygf}
        [Schottky Theorem].
        Suppose $f(z)$ is holomorphic in $\Omega\subset \Cn$, $n\geq 1$, which range omits $0$ and $1$ there.
        If $a\in \Omega$ and $|f(a)|<M$, then for every $\varepsilon>0$ sufficiently small we have
        \begin{equation}
            \label{eq1}
            |f(z)|<C \textrm{ for all }z\in \Omega, \ \ \delta(z)\geq\varepsilon,
        \end{equation}
        where $\delta(z)$ denotes the Euclidean distance from $z\in \Omega$ to $\partial\Omega$.
        Here $C$ is a positive constant depending only on $f(a)$ and $\varepsilon$.
    \end{thm}

    The proof of Theorem \ref{shottkygf} is based on the following lemma.
    \begin{lem}
        \label{lem}
        Let $\F$ be a normal family of holomorphic functions in a domain $\Omega\subset\Cn$, $n\geq 1$.
        Let $M$ be a positive number and $a\in \Omega$.
        Assume that for each function $f\in \F$, we have
        \[
            \abs{f(a)}<M.
        \]
        Then the family $\F$ is uniformly bounded on each compact subset of $\Omega$.
    \end{lem}
    \begin{proof}
        Let $K$ be a compact set of points belonging to $\Omega$.
        Assume that $\F$ is not uniformly bounded on $K$.
        Then to each positive integer $k$, corresponds a function $f_k\in \F$ such that
        \[
            \max_{z\in K}\abs{f_k(z)}>k.
        \]
        The sequence $\{f_k\}$ can not converge on compact $\{a\}\cup K$ to a holomorphic function or to $\infty$ in $\Omega$.
        This contradicts the hypothesis that $\F$ is a normal family.
    \end{proof}

    \begin{proof}
        [Proof of Theorem \ref{shottkygf}] Let $\mathbf{G}$ denote the family of all functions holomorphic on $\Omega$ which omit the values $0$ and $1$.
        According to Montel's theorem \cite[Theorem 1.4]{[DP1]} the family $\mathbf{G}$ is normal.
        By Lemma \ref{lem}, the family $\mathbf{G}$ is uniformly bounded on $\overline{\Omega}_\varepsilon=\set{z\in \Omega, \, \delta(z)\geq \varepsilon}$, hence there is a positive constant $C$ such that (\ref{eq1}) holds for each function $f(z)$ of $\F$.
        This constant $C$ has then the required property.
    \end{proof}

    The method of proof in Theorem \ref{shottkygf} does not give any information about what the constant $C$ is.

    \begin{rem}
        In the case $n=1$ the Shottky theorem has numerous proofs.
        We refer to \cite{[RBB]} for an exposition (the history, methods and references) of the theorem.
        Schottky's original theorem \cite{[FS]} did not give an explicit bound for $f$.
        Let $K(f(0), r)$ denote the best possible bound in Theorem \ref{shottkygf}.
        Various authors have dealt with the problem of giving an explicit estimate for this bound see Jenkins \cite{[JAJ]}; Hempel \cite{[JAH]} gave some bounds whose constants are in some sense the best possible.
    \end{rem}

    It has been proved by Murali Rao \cite[Theorem 7.17, p. 118]{[RAO]} that Shottky's Theorem holds with
    \[
        C=\tan(\arctan M+(c(1-r^2))^{-1}),
    \]
    where $c$ is some constant.
    The proof is based on the work of Minda and Schober \cite{[DMGS]}.
    We shall give an alternative argument based on the classical Montel Theorem, and Marty's criterion for normality.

    Recall that to every $a$ in the unit disc $\Delta:=\set{z\in \C : \abs{z}<1}$ corresponds an automorphism $\varphi_a$ of the disc that interchanges $a$ and 0, namely $\varphi_a(w):=(w-a)/(1-\overline{a}w)$.
    Note for later reference that $\varphi_a(0)=-a$ and $\varphi_a'(0)=1-\abs{a}^2$.
    Define $Aut(\Delta):= \set{\varphi_a, a\in \Delta}$.

    The following theorem is the main result of this note.
    \begin{thm}
        \label{ch2:thm:schottky1}
        (see \cite[Theorem 7.17, p. 118]{[RAO]}) Suppose $f(z)$ is holomorphic in the unit disc $\Delta$ whose range omits $0$ and $1$.
        If $|f(0)|<M$, then for every $r\in (0, 1)$ we have
        \begin{equation}
            \label{eq2}
            |f(z)|<\tan(\arctan M+2rL/(1-r^2)) \textrm{ for all }\abs{z}<r, r\in \interval[open right]{0}{1}.
        \end{equation}
    \end{thm}
    \begin{proof}
        Let $\mathbf{G}$ denote the family of all functions holomorphic on the unit disc $\Delta$ which omit the values $0$ and $1$.
        By Montel's Theorem \cite[p. 218]{[LZ]} the family $\mathbf{G}\circ Aut(\Delta)$ is normal.
        Marty's Theorem \cite[p. 216]{[LZ]} yields a non-negative constant $L$ such that for any $f \circ \varphi_a\in \mathbf{G}\circ Aut(\Delta)$
        \[
            \frac{\abs{f'(\varphi_a(0))}}{1+\abs{f(\varphi_a(0))}^2}\leq L.
        \]
        The constant $L$ in the inequality above does not depend on $f\circ \varphi_a$.
        The exact value $L$ is not known.

        Using the chain rule for differentiation and replacing $-a$ by $z$ we obtain
        \[
            \frac{\abs{f'(z)}(1-\abs{z}^2)}{1+|f(z)|^2}\leq L \textrm{ for all }z\in \Delta.
        \]
        Since $f$ never takes the value $0$, the function $t \to |f(tz)|$ is continuously differentiable on $\interval[open]{0}{1}$ for any fixed $z \in \Delta[r]=\set{\abs{z}<r}$ and
        \[
            \frac{d}{dt}\Big(\arctan(|f(tz)|)\Big) \leq \frac{2\abs{f'(tz)}\abs{z}}{1+|f(tz)|^2}\leq\frac{2rL}{1-r^2}.
        \]
        Integrating this from $t = 0$ to $1$, we deduce
        \[
            \Big|\arctan(|f(z)|)-\arctan(|f(0)|)\Big|\leq\int_0^{1}\Big|\frac{d}{dt}(\arctan(|f(tz)|))\Big|dt\leq\frac{2rL}{1-r^2}
        \]
        and hence
        \[
            \arctan(|f(z)|)\leq\arctan(|f(0)|)+\frac{2rL}{1-r^2}<\arctan M+\frac{2rL}{1-r^2}
        \]
        which is the theorem with $C=\tan(\arctan M+2rL/(1-r^2))$.
        This completes the proof of Theorem \ref{ch2:thm:schottky1}.
    \end{proof}
    \begin{rem}
     The same proof works for $n>1$.
    \end{rem}

    This type of argument can also be used to prove:

    \begin{thm}
        \label{liuville}
        [Picard Little Theorem].
        At most one complex number is absent from the range of a non-constant entire function ($=$ holomorphic in $\C$).
    \end{thm}
    \begin{proof}
        Assume, to the contrary, that an entire function $f$ omits two distinct values from its range.
        Composing with a linear fractional transformation, we may also assume that the omitted values are $0$ and $1$.

        The function $f(2^k\lambda)$, where $k$ is positive integer, is holomorphic in $\Delta$, $f(2^k\lambda)$ takes the same values in $\Delta$ as does $f$ in $\abs{z}<2^k$, and $f(2^k\lambda)$ omits the values $0$ and $1$.

        The family $\set{f(2^k\varphi_a), \varphi_a\in Aut(\Delta), k=1, 2 \ldots }$, being contained in the family of functions avoiding $0$ and $1$, is normal in $\Delta$ by Montel's theorem \cite[p. 218]{[LZ]}.
        By Marty's Theorem \cite[p. 216]{[LZ]}, there exists $L > 0$ such that for any $f(2^k\varphi_a)$
        \[
            \frac{\abs{f'(2^k\varphi_a(0))}}{1+\abs{f(2^k\varphi_a(0))}^2}\leq L.
        \]
        After some manipulations, using the chain rule for differentiation, the equalityes $\varphi_a(0)=-a$ and $\varphi'_a(0)=1-\abs{a}^2$, we find
        \[
            \frac{2^k\abs{f'(-2^ka)}(1-\abs{a}^2)}{1+\abs{f(-2^ka)}^2}\leq L.
        \]
        Replacing in this inequality $-2^{k}a$ by $z$ we get
        \[
            \frac{2^k\abs{f'(z)}(1-\abs{z/2^k}^2)}{1+\abs{f(z)}^2}\leq L.
        \]
        Thus
        \[
            \frac{\abs{f'(z)}}{1+|f(z)|^2}\leq \frac{2^{k}L}{2^{2k}-\abs{z}^2}.
        \]
        Keeping $z$ fixed, let $k$ tend to infinity to deduce that $f'(z)=0$ for every point $z$ in the complex plane.
        But this clearly implies that $f(z)$ is a constant, not the non-constant entire function.
        The obtained contradiction proves the theorem.
    \end{proof}
    Immediate (and easy) corollary of Picar's Little Theorem is Liouville's Theorem:
    \begin{cor}
        Every bounded entire function is constant.
    \end{cor}

    Theorem \ref{liuville} is a remarkable generalization of Liouville's theorem.
    It is simple to find entire functions whose range is the entire $\C$; nonconstant polynomials, for instance.
    The exponential function is an example of an entire function whose range omits only one value, namely zero.
    But there does not exist a nonconstant entire function whose range omits two different values.

    The fundamental theorem of algebra is a consequence of Liouville's Theorem (for the proof see, for example, \cite[Corollary 4, p. 98]{[RAO]}).

    An alternate program for obtaining the above theorems (and not only) in elementary fashion was indicated by Zalcman \cite[p. 817]{[LZ1]}.

    \begin{rem}
        Many of the theorems on the theory of normal families for holomorphic functions of one variable may be transferred without essential changes to the case of functions of several complex variables.

        Slice functions is a useful tool that allow us to apply facts from the function theory of unit disc $\Delta$ to questions in the unit ball $B\subset \Cn$.
        Suppose $f(z)$ is holomorphic in the unit ball  whose range omits $0$ and $1$.
        Fix a point $z\in B$ with nonzero norm $r=\abs{z}$ and consider the holomorphic function of one complex variable that sends a number $\lambda$ in the unit disk to the image $f(\lambda z/r)$.
        This slice function omits the values $0$ and $1$, takes the value $f(0)$ when $\lambda= 0$, and takes the value $f(z)$ when $\lambda=r$.
        Then $\abs{f(z)}<C$, where $C$ is the constant in Theorem \ref{ch2:thm:schottky1} corresponding to $f(0)$ and $r$.
        Thus Schottky's theorem holds in the unit ball  with the same constant as in dimension one.
    \end{rem}

    

\begin{thebibliography}{20}
        \bibitem{[FS]} F.\ Schottky.\ ``On Picard's theorem and Borel's inequality''.\ German.
        In: Berl.\ Ber.\ 1904 (1904), pp.~1244--1263.

        \bibitem{[RBB]} Robert B.\  Burckel.\  Classical analysis in the complex plane.\
 Cornerstones.\  Springer,\ New York, \ $[2021]$ , pp.~xxix+1123. isbn: 978-1-0716-1963-6; 978-1-0716-1965-0. doi: 10.1007/978-1-0716-1965-0.

        \bibitem{[JAJ]} J.\ A.\ Jenkins.\ ``On explicit bounds in Schottky's theorem''.
        In: Canadian J.\ Math.\ 7 (1955), pp.~76--82.\ issn: 0008--414X.\ doi: 10.4153/CJM-1955-010-4.

        \bibitem{[JAH]} Joachim A.\ Hempel.\ ``Precise bounds in the theorems of Schottky and Picard''.
        In: J.\ London Math.\ Soc.\ (2) 21.2 (1980), pp.~279--286.\ issn: 0024--6107.\ doi: 10.1112/jlms/s2-21.2.279.

        \bibitem{[LZ]} Lawrence Zalcman. ``Normal families: new perspectives''.
        In: Bull.
        Amer.
        Math.
        Soc.
        (N.S.) 35.3 (1998), pp. 215–230. issn: 0273-0979. doi: 10.1090/S0273-0979-98-00755-1.

        \bibitem{[LZ1]} Lawrence Zalcman. ``A heuristic principle in complex function theory''.
        In: Amer.
        Math.
        Monthly 82.8 (1975), pp. 813–817. issn: 0002-9890. doi: 10.2307/2319796.

        \bibitem{[DP1]} P.\ V.\ Dovbush.\ ``On normal families in $C^n$''.
        In: Complex Var.\ Elliptic Equ.\ 67.1 (2022), pp.~1--8.\ issn: 1747--6933.\ doi: 10.1080/17476933.2020.1797703.

        \bibitem{[RAO]} Murali Rao \ et \ al.
        \ Complex analysis.\ (2nd Edition).
        \ World Scientific Publishing Co.\ Pte.\ Ltd., Hackensack, NJ, 2015, pp.
        \ x+414.
        \ isbn: 978--981--4579--599.\ doi: 10.1142/9062.

        \bibitem{[DMGS]} David Minda and Glenn Schober.\ ``Another elementary approach to the theorems of Landau, Montel, Picard and Schottky''.
        In: Complex Variables Theory Appl.\ 2.2 (1983), pp. 157--164.\ issn: 0278--1077.\ doi: 10.1080/17476938308814039.
    \end{thebibliography}
\end{document}